\documentclass[12pt]{amsart}

\usepackage{amssymb,amsfonts}

\usepackage{amsthm}
\usepackage[dvipsnames]{xcolor}
\usepackage{graphicx}
\usepackage{xfrac}
\usepackage{tikz-cd}
\usepackage[all,arc]{xy}
\usepackage[colorlinks=true, allcolors=blue]{hyperref}
\usepackage[only,llbracket,rrbracket]{stmaryrd}
\usepackage{enumerate}
\usepackage{mathrsfs}
\usepackage{tikz-cd}
\usepackage{import}
\usepackage[utf8]{inputenc}
\usepackage{hyperref}
\usepackage{cancel}
\usepackage{bbm}
\usepackage{comment}
\usepackage[left=1in,top=1in,right=1in,bottom=1in]{geometry}
\usepackage{mathtools}

\usepackage{amsthm}
\theoremstyle{plain}
\newtheorem{theorem}{Theorem}[section]
\newtheorem{lemma}[theorem]{Lemma}
\newtheorem{prop}[theorem]{Proposition}
\newtheorem{cor}[theorem]{Corollary}

\theoremstyle{definition}
\newtheorem{definition}[theorem]{Definition}
\newtheorem{definition-lemma}[theorem]{Definition/Lemma}
\newtheorem{example}[theorem]{Example}

\newtheorem{observation}[theorem]{Observation}

\theoremstyle{remark}

\newtheorem{notation}[theorem]{Notation}

\newcommand{\longsquigarrow}{\xymatrix{{}\ar@{~>}[r]&{}}}

\newcommand{\Z}{\mathbb{Z}}

\newcommand{\R}{\mathbb{R}}

\newcommand{\Hom}{\text{Hom}}

\newcommand{\into}{\hookrightarrow}
\newcommand{\onto}{\twoheadrightarrow}

\renewcommand{\H}{\mathrm{H}}
\renewcommand{\Bar}{\operatorname{Bar}}
\newcommand{\Cyc}{\operatorname{Cyc}}

\newcommand{\desus}{\mathbf{s}^{-1}}
\newcommand{\sus}{\mathbf{s}}

\newcommand{\maps}{\colon}

\begin{document}

\noindent
\title{
  The cyclic bar construction and fundamental groups
  }
\begin{abstract}
We determine the $0$-th Hochschild homology of the associative algebra of simplicial cochains valued in a PID: it consists of the ``finite-type" homotopy invariants of free loops, equivalently finite-type class functions on the fundamental group. One major motivation for this calculation is joint work in progress aiming to geometrically construct invariants of links in the $3$-sphere as well as other $3$-manifolds, and to realize Milnor's linking numbers as evaluations of $0$-th Hochschild homology classes.
\end{abstract}

\author{Nir Gadish}
\address{Department of Mathematics, University of Pennsylvania, Philadelphia, PA}
\email{\url{ngadish@math.upenn.edu}}


 
\maketitle

\section{Introduction}
Given an associative augmented dg-algebra $(A^\bullet,d)$, its cyclic Bar construction, or Hochschild complex, is the complex of tensors $\Cyc(A):= A\otimes T (\bar{A})$, where $T(-)$ is the tensor algebra functor applied to (the degree shift by $-1$ of) the augmentation ideal $\bar{A}$. This has natural differential $d_{\Cyc}$ of algebraic and geometric significance, so that for any \underline{simply connected} topological space $Y$ with singular cochain algebra $A = C^*_{Sing}(Y;\Z)$ there is a well-known natural isomorphism
\[
\H^*(\Cyc(A),d_{\Cyc}) \cong \H^*(LY;\Z)
\]
where $LY = \operatorname{Maps}(S^1,M)$ is the free loop space of $Y$ (see e.g. \cite{jones_cyclic_1987,loday_free_2011}). Alternatively, when $Y$ is a simply connected smooth manifold and $A$ is its deRham algebra $\Omega^*_{dR}(Y)$, an analogous isomorphism holds with the $\R$-valued cohomology of $LY$, realized by Chen's iterated integrals \cite{getzler_differential_1991}.

These examples leave open the question of what $\Cyc(A)$ is calculating when $Y$ is connected but not simply connected. In that case, the connected components $\pi_0(LY)$ are indexed by the set of conjugacy classes of $\pi_1(Y)$, so one might expect $\H^0(\Cyc(A))$ to compute the set class functions, that is, conjugation-invariant functions $\pi_1(Y)\to \Z$. This optimistic guess fails, but not too badly, as our main result shows.

Below, we state results for simplicial sets. This setup effectively includes all topological spaces via the singular simplicial set construction $Y \mapsto \operatorname{Sing}_\bullet(Y)$, but also includes finite combinatorial models of spaces and groups of interest in applications. In particular, it is possible to implement the isomorphisms below algorithmically for finitely presented groups and finite 2-complexes.

Let $R$ be any PID, e.g. $R = \Z, \R$ or $\Z/m\Z$. If $G$ is a group, then the group algebra $R[G]$ has a natural $G$-action by conjugation. Its \emph{augmentation ideal} $I$, generated by differences $\{g-1 \mid g\in G \}$ is invariant under conjugation by $G$. A function $G\to R$ is said to be of \emph{finite type} $n$ is its $R$-linear extension $R[G]\to R$ annihilates $I^{n+1}$, e.g. constants are of type $0$ and homomorphisms are of type $1$.
\begin{theorem}\label{thm-main}
    Let $Y$ be a connected pointed simplicial set, with finitely generated fundamental group $\pi_1(Y)$. Let $C^*_{\Delta}(Y;R)$ denote the $R$-valued simplicial cochains algebra with its standard Alexander--Whitney cup product, and $\Cyc_{\Delta}(Y;R)$ the associated cyclic Bar complex (a.k.a. the Hochschild complex, defined below).
    Then there is a natural isomorphism of filtered $R$-modules
    \begin{equation}\label{eq:main iso}
    \H^0(\Cyc_{\Delta}(Y;R)) \xrightarrow{\sim} \varinjlim_{n} \Hom_{R}\left( R[\pi_1(Y)]/I^n , R\right)^{\pi_1(Y)}
    \end{equation}
    where the target is the module all finite type functions $\pi_1(Y)\to R$ that are furthermore invariant under conjugation by $\pi_1(Y)$, that is, $R$-valued class functions of finite type.
    \end{theorem}

For an algebraically presented group, the isomorphism in \eqref{eq:main iso} is given explicitly by \emph{letter-braiding}: if $G=F_S$ is the free group on elements $s\in S$, and $h_1,\ldots,h_n\maps F_S\to R$ is an $n$-tuple of group homomorphisms, we define in \cite{gadish_letter-braiding_2023} the letter-braiding invariant of $\bar{h}=[h_1|\ldots|h_n]$ to be the function $I_{\bar{h}}\maps F_S\to R$ given as follows.
    For $w = s_1 s_2 \cdots s_k$ a word in generators and their inverses, i.e. $s_i \in S\cup S^{-1}$,
    \begin{equation}\label{eq:letter braiding}
        I_{\bar{h}}\maps w \longmapsto \sum_{1\leq i_1 <_w \, i_2 <_w  \,\ldots <_w \, i_n \leq k} h_1(s_{i_1}) h_1(s_{i_2}) \cdots h_n(s_{i_n}),
    \end{equation}
    where $i<_w \, j$ is the order relation
    \[
    i<_w j = \begin{cases}
        i<j & s_i\in S \\
        i\leq j & s_i\in S^{-1}.
    \end{cases}
    \]
    This formula realizes the map in \eqref{eq:main iso} when $Y$ is a wedge of $S$-many simplicial circles: in this case the $1$-cocycles $C^1_{\Delta}(Y;R)$ are the set of functions $S\to R$, equivalent to homomorphisms $F_S\to R$. But since every group has a presentation $F_S\onto G$, the naturality of the ismorphism in \eqref{eq:main iso} shows that it is uniquely characterized by letter-braiding of generators as in \eqref{eq:letter braiding}. 

    In a prequel to this work \cite{gadish_letter-braiding_2023}, we showed that letter-braiding determines a natural isomorphism between the ordinary (noncyclic) Bar construction $\Bar(A) := T(\desus\bar{A})$ and the set of all functions $\varinjlim\Hom_R (R[\pi_1(Y)]/I^n, R)$, without restricting to conjugation-invariant functions. The two flavors of Bar construction are related naturally by the identification
    \begin{equation}\label{eq:identification of cyc and bar}
    1\otimes [ a_1\otimes \ldots \otimes a_n ]\in \Cyc(A) \longleftrightarrow [a_1\otimes \ldots \otimes a_n] \in \Bar(A),
    \end{equation}
    which is an isomorphism in cohomological degree $0$, although their differentials differ. This extends to a chain map $\Cyc(A)\to \Bar(A)$ by applying the augmentation $A\to R$ to the left-most term of $A\otimes T(\bar{A})$.

    For a pointed simplicial set $Y$ denote $\Bar_{\Delta}(Y;R) = \Bar(C^*_{\Delta}(Y;R))$. We abbreviate $\H^*_{\Cyc}(Y;R) := \H^*(\Cyc_{\Delta}(Y;R))$ and $\H^*_{\Bar}(Y;R) := \H^*(\Bar_{\Delta}(Y;R))$.
    \begin{theorem}\label{thm:main pullback}
    Let $Y$ be connected pointed simplicial set. Then the augmentation $\Cyc_{\Delta}(Y;R)\to \Bar_{\Delta}(Y;R)$ from \eqref{eq:identification of cyc and bar} defines a pullback square of injections through letter braiding invariants
    \[
    \xymatrix{
    \H^0_{\Cyc}(Y;R) \ar@{^(->}[r] \ar@{^(->}[d] & \H^0_{\Bar}(Y;R) \ar@{^(->}[d] \\
    \operatorname{Func}(\pi_1(Y),R)^{\pi_1(Y)} \ar@{^(->}[r] & \operatorname{Func}(\pi_1(Y),R).
    }
    \]
    That is, the letter braiding invariants arising from the cyclic complex are exactly those letter braiding invariants that happen to also be conjugation-invariant.
    \end{theorem}

It is furthermore possible to determine whether a letter braiding invariant is a class function using the \emph{cycle operator}, which Connes used to define his cyclic homology complex.
    \begin{theorem}\label{thm:topological characterization}
        For any $R$-module $V$ let $\sigma\maps T(V)\to T(V)$ be the natural \emph{cycle} automorphism of the tensor algebra that cycles tensors once to the right (with a sign given in \ref{def:cycle operator}):
        \[
        \sigma(v_1\otimes v_2\otimes  \ldots v_n) = \pm v_n\otimes v_1 \otimes \ldots \otimes v_{n-1}.
        \]

    Then for a pointed connected simplicial set $Y$ and a set of group generators $S\subseteq \pi_1(Y)$, the module $\H^0_{\Cyc}(Y;R)$ is identified with the collection of noncommutative polynomials in variables $h\in R^S$ that are
    \begin{itemize}
        \item cycle-invariant, and
        \item the evaluation on words $w\in F_S$ according to Equation \eqref{eq:letter braiding} descends to a well-defined function on $\pi_1(Y)$.
    \end{itemize}
    In other words, there is a pullback diagram of injections
\begin{equation}
\begin{split}
    \xymatrix{
    \H^0_{\Cyc}(Y;R) \ar[r] \ar[d] & \left(\bigoplus_{p\geq 0} (R^S)^{\otimes p} \right)^\sigma \ar[d] \\
    \operatorname{Func}(\pi_1(Y),R) \ar[r] & \operatorname{Func}(F_S,R).
    }
\end{split}    
\end{equation}
\end{theorem}
Combinatorially, the last theorem states that a letter-braiding invariant is a class function if and only if it does not change under cyclic rotation of the functionals
\[ [h_1|\ldots|h_n] \longmapsto [h_n|h_1|\ldots|h_{n-1}].\]

\subsection*{Acknowledgements} This project grew out of many fruitful conversations with Greg Friedman, Robin Koytcheff, Dev Sinha, and Ben Walter. I thank Dev Sinha and Greg Friedman for their thoughtful feedback on an early draft.

\section{Associative algebras}
Let $R$ be a ground PID and consider $(A,d)$ an associative differential graded (dg-) $R$-algebra, equipped with augmentation
\[
0 \to \bar{A} \to A \xrightarrow{\eta} R \to 0.
\]
We call $\bar{A}$ the augmentation ideal. The following terminology is standard.
\begin{definition}
$A$ is said to be \emph{connected} if the augmentation induces and isomorphism $A^0\cong R$, or equivalently if $\bar{A}$ is concentrated in strictly positive degree. More generally, $A$ is said to be \emph{homologically connected} if it is quasi-isomorphic to a connected dg algebra.
\end{definition}

\begin{definition}[Cyclic Bar construction]\label{def:cyclic bar}
    Let $(M,d_M)$ be a differential graded $(A-A)$-bimodule, meaning that $d_M$ satisfies the (graded) Leibniz rule with respect to multiplication by $A$ on either side. The cyclic Bar complex of $A$ with coefficients in $M$ is the graded module
    \[
    \Cyc(A;M) := M\otimes T(\desus \bar{A})
    \]
    with tensor notation abbreviated to
    \[
    m_0[a_1|\ldots|a_n] \sim m_0\otimes \desus a_1 \otimes \ldots \otimes \desus a_n,
    \]
    and differential $d (m_0[a_1|\ldots|a_n])= d_M(m_0)[a_1|\ldots|a_n] - \sum_{i=1}^n (-1)^{\epsilon_{i-1}}m_0[a_1|\ldots|da_i|\ldots|a_n]$
    \begin{eqnarray*}
     - (-1)^{|m_0|} & m_0  a_1 [a_2|\ldots | a_n]&  -
     \sum_{i=1}^{n-1}(-1)^{\epsilon_i} m_0[a_1|\ldots|a_i a_{i+1}|\ldots |a_n] \\+ (-1)^{\epsilon_{n-1}(|a_n|-1)} & a_nm_0[a_1|\ldots|a_{n-1}]
    \end{eqnarray*}
    where $\epsilon_i = |m_0|+|a_1|+\ldots+|a_i|-i$ (we follow the sign conventions of \cite[Section 2]{getzler_differential_1991}).

    In the special case $M=A$ we denote $\Cyc(A;A)$ simply by $\Cyc(A)$.
\end{definition}

\begin{example}[Linear Bar construction]\label{obs:bar is cyclic}
    The augmentation $\eta\maps A\to R$ endows $R$ with a trivial $A$-bimodule structure, and the cyclic complex $\Cyc(A;R)$ is naturally isomorphic to the classical (linear) bar construction  $\Bar(A)$. This may be taken as a definition:
    \[
    \Cyc(A;R) = R\otimes T(\desus\bar{A}) \cong T(\desus \bar{A}) =: \Bar(A)
    \]
    consists of sums of elements $[a_0|\ldots|a_n]$, with differential $d_{\Bar}$ consisting only of those terms of $d_{\Cyc}$ that don't operate on $m_0$. This is because the terms in which $\bar{A}$ is multiplied with the module $M=R$ through the augmentation automatically vanish.
\end{example}

\begin{observation}\label{obs:bar and cyc(a-bar) are the same}
    The cyclic complex $\Cyc(A;\bar{A}) = \bar{A}\otimes T(\desus \bar{A})$ is additively isomorphic to the positive degree part of the tensor algebra $T^{>0}(\bar{A})$. Accordinely, there is a natural additive isomorphism of degree $+1$
    \[
    \iota\maps \overline{\Bar(A)} \to \Cyc(A;\bar{A}) \quad\text{ given by }\quad [a_1|\ldots|a_n] \longmapsto a_1[a_2|\ldots|a_n].
    \]
    This will play a role in our analysis below.
\end{observation}

Crucially, the bi-functor $\Cyc(-;-)$ is quasi-isomorphism invariant, in the following sense.
\begin{prop}
    Let $A$ and $A'$ be augmented dg-algebras that are $R$-torsion free, and let $M$ and $M'$ be bounded-below bimodules over the respective algebras. If $h\maps A\to A'$ is a quasi-isomorphism of dg-algebras and $g\maps M\to M'$ is a quasi-isomorphism of bimodules, compatible with $h$, then the induced map
    \[
    g\otimes T(h) \maps \Cyc(A;M) \to \Cyc(A';M')
    \]
    is a quasi-isomorphism. In particular, the induced map $\H^i(\Cyc(A;M))\to \H^i(\Cyc(A';M'))$ is an isomorphism for all $i$.
\end{prop}
\begin{proof}
    Use the standard spectral sequence argument: the complexes are filtered by tensor degree, with associated graded breaking up as a direct sum of tensor product complexes 
    \[
    E_0 \cong \bigoplus_{p\geq 0} (M\otimes (\desus \bar{A})^{\otimes p}, d).\]
    The K\"unneth isomorphism applies, since $\desus\bar{A}$ are $R$-torsion free and $R$ is a PID, and we get that the $E_1$-page is naturally isomorphic to $\Cyc(H(A);H(M))$. The same holds for $(A',M')$, so since the maps of complexes $(A,M)\to (A',M')$ induce isomorphisms on cohomology, they already induce an isomorphism on the $E_1$-pages of the spectral sequence. These spectral sequences converge since $\Cyc(A;M)$ is bounded-below, so a map inducing an isomorphism on $E_1$ is a quasi-isomorphism.
\end{proof}

    We will use Connes' cyclic rotation operator on the tensor algebra. 
    \begin{definition}\label{def:cycle operator}
    Let $V$ be any graded $R$-module. Then the tensor algebra $T(V)$ has a natural automorphism
    \[
    \sigma\maps (v_1\otimes \ldots\otimes v_n) = (-1)^{ (|v_1|+\ldots+|v_{n-1}|)\cdot |v_n|}v_n\otimes v_1\otimes \ldots \otimes v_{n-1}.
    \]
    We call this the \emph{cycle} operator.

    In particular, $\sigma$ is defined on the linear Bar construction $\Bar(A) = T(\desus \bar{A})$.
    \end{definition}

    \begin{observation}
        The cycle operator on $\Bar(A)$ is \emph{not} an automorphism of the chain complex, as it does not commute with the Bar differential.
        However, the submodule $\Bar(A)^{\sigma}$ of $\sigma$-invariant tensors happens to be preserved by the Bar differential and is thus a subcomplex. The reader should keep in mind that, while $\Bar(A)$ is a quasi-isomorphism invariant of $A$, the subcomplex $\Bar(A)^{\sigma}$ is not, and may even exhibit nontrivial cohomology in strictly negative cohomological degrees.
    \end{observation}

\begin{lemma}\label{lem:comparing cyclic and bar}
    Let $A$ be an augmented homologically connected dg-algebra without $R$-torsion. The augmentation sequence $0\to \bar{A}\to A\to R\to 0$ induces a long exact sequence in cyclic Bar cohomology
    \begin{equation}
        0\to H^0(\Cyc(A)) \to H^0(\Bar(A)) \overset{d}\to H^1(\Cyc(A;\bar{A})\to \H^1(\Cyc(A)) \to \ldots
    \end{equation}
    whose connecting homomorphism
    $d\maps H^i(Bar(A)) \to H^{i+1}(Cyc(A;\bar{A}))$
    is $d = \iota \circ(\sigma-1)$.
\end{lemma}
\begin{proof}
First, by quasi-isomorphism invariance of the cyclic complexes, we may assume $A$ is a connected algebra, so $\bar{A}$ is concentrated in strictly positive degrees. Consequently, the three complexes $\Cyc(A;M)$ for $M=R,A,$ and $\bar{A}$ are trivial in strictly negative degrees.

The augmentation induces a short exact sequence of cyclic chain complexes
\[
0\to \Cyc(A;\bar{A})\to \Cyc(A;A) \to \Cyc(A;R) \to 0,
\]
which in turn gives a long exact sequence in cohomology
\[
\ldots \to H^i(\Cyc(A))\to H^i(\Cyc(A;R))\xrightarrow{d} H^{i+1}(\Cyc(A;\bar{A}) \to H^{i+1}(\Cyc(A)) \to \ldots.
\]
By Example \ref{obs:bar is cyclic}, the complex $\Cyc(A;R)$ is the same as the linear Bar complex of $A$. Furthermore, the complex $\Cyc(A;\bar{A})$ is concentrated in strictly positive degrees, and thus has vanishing $0$-th cohomology.

It remains to compute the connecting homomorphism. By definition
\begin{align*}
&d \sum_{i=1}^N[a_1|\ldots|a_{n_i}] = d_{\Cyc} \sum_{i=1}^N 1[a_1|\ldots|a_{n_i}] \\ = &-\sum_{i=1}^N a_1[a_2|\ldots|a_{n_i}] + 1\otimes d_{\Bar}[a_1|\ldots|a_{n_i}] + (-1)^{\epsilon_{n_{i}-1}(|a_{n_i}|-1)}a_{n_i}[a_1|\ldots|a_{n_i-1}].
\end{align*}
But since this operator is only applied to cocycles in the Bar complex, the $d_{\Bar}$ term vanishes and we are left with exactly $\iota\circ (\sigma-1)$. The sign on the $\sigma$ term is verified since the degree of every $\desus a_i$ factor in the tensor is $|a_i|-1$.
\end{proof}

\begin{example}\label{ex:wedge of circles}
    Consider the simplicial cochain algebra of a wedge of circles $X_n = \bigvee_{i=1}^n S^1$. This is the square-zero dg algebra $A_n = R[\epsilon_1,\ldots,\epsilon_n]/(\epsilon_i \epsilon_j \mid 1\leq i,j \leq n)$ with trivial differential, and augmentation $\epsilon_i \mapsto 0$ for all $i$.

    Then the corresponding Bar complex is the (co)free (co)associative (co)algebra $$B_n = R\langle e_1,\ldots,e_n \rangle \quad\text{ where }\quad e_i = \desus(\epsilon_i)$$ concentrated in degree $0$ equipped with zero differential and the cycle operator $\sigma$.

    In this case, 
    \[
    \H^*(\Bar(A_n)) = \begin{cases}
        B_n & *=0\\
        0 & *\neq 0
    \end{cases} \quad\text{ and } \quad \H^*(\Cyc(A_n;\bar{A_n})) = \begin{cases}
        \bar{B}_n & *=1 \\
        0 & *\neq 1,
    \end{cases}
    \]
    where $\bar{B}_n$ is the augmentation ideal consisting of polynomials with no constant term.
    
    The long exact sequence in Lemma \ref{lem:comparing cyclic and bar} gives
    \[
    0 \to H^0(\Cyc(A_n)) \to B_n \xrightarrow{\sigma-1} \bar{B}_n \to H^1(\Cyc(A_n)) \to 0 \to \ldots\to 0 \to \H^i(\Cyc(A_n)) \to 0
    \]
    translating to isomorphisms with cycle-invariants and coinvariants
    \[
    H^*(\Cyc(A_n)) \cong \begin{cases}
        (B_n)^{\sigma} & *=0 \\
        (\bar{B}_n)_{\sigma} & *=1\\
        0 & *\geq 2.
    \end{cases}
    \]
\end{example}
\begin{notation}\label{not:infinite wedge}
    Below we consider infinite wedges of circles, so let us set notation analogous to the last example. For any index set $S$ let $X_S := \bigvee_{s\in S} S^1$ be the wedge of $S$-many circles. Its fundamental group is the free group $F_S = \langle S\rangle$, generated by $S$.

    Simplicial chains on $X_S$ have $C^0(X_S) = R\cdot 1$ and $C^1(X_S) = \{ h\maps S\to R \} =: R^S$ with trivial product and differential. Let us denote its Bar construction by $B_S = T(\desus R^S)$, concentrated in degree $0$ and having zero differential. It is spanned by tensors $[h_1|\ldots|h_n]$, as in Equation \eqref{eq:letter braiding}, but note that it is uncountably generated when $S$ is infinite.
\end{notation}
The identification of $\H^0(\Cyc(A))$ with cycle-invariant tensors in fact holds more generally.

    \begin{lemma}
        The natural injection $\tau\maps \Bar(A) \to \Cyc(A)$ given by $[a_1|\ldots|a_n]\mapsto 1[a_1|\ldots|a_n]$ restricts to a chain map on the $\sigma$-invariant subcomplex,
        \[
        \tau\maps (\Bar(A)^\sigma, d_{\Bar}) \to (\Cyc(A),d_{\Cyc}).
        \]
    \end{lemma}
    \begin{proof}
        For $n>0$ we compute
        \[
        d_{\Cyc} 1[a_1|\ldots|a_n] = -a_1[a_2|\ldots|a_n] + 1\otimes d_{\Bar}[a_1|\ldots|a_n] + (-1)^{\epsilon_{n-1}\cdot(|a_n|-1)}a_n[a_1|\ldots|a_{n-1}],
        \]
        with the two extremal terms landing in $\Cyc(A;\bar{A})$, where they happen to lie in the image of the isomorphism $\iota\maps \overline{\Bar(A)}\to \Cyc(A;\bar{A})$ from Observation \ref{obs:bar and cyc(a-bar) are the same}.
        In fact, we have
        \[
        \iota(\sigma-1) [a_1|\ldots|a_n] = -a_1[a_2|\ldots|a_n] + (-1)^{\epsilon_{n-1}(|a_n|-1)} a_n[a_1|\ldots|a_{n-1}]
        \]
        so that overall there is the relation
        \[
        d_{\Cyc}\tau = \iota(\sigma-1) + \tau d_{\Bar}.
        \]
        In particular, since $\sigma$-invariant tensors vanish on $(\sigma-1)$, the restriction of $\tau$ to those tensors is a chain map.
    \end{proof}

\begin{prop}\label{prop:pullback square connected}
    If $A$ is any connected dg algebra, the $0$-th cyclic cohomology of $\Cyc(A)$ is exactly the cycle-invariant subspace of the $0$-th Bar cohomology. That is, the natural map $\tau\maps \Bar(A)^\sigma \to \Cyc(A)$ induces a natural isomorphism $\H^0(\Bar(A)^\sigma) \xrightarrow{\sim} \H^0(\Cyc(A))$, and there is a natural pullback square of injections
    \[
    \xymatrix{
    \H^0(\Cyc(A)) \ar@{^(->}[r] \ar@{^(->}[d] & \Bar(A)^\sigma \ar@{^(->}[d]\\
    \H^0(\Bar(A)) \ar@{^(->}[r] & \Bar(A)
    }
    \]
where all but the top arrow are the obvious maps.
\end{prop}
\begin{proof}
    Since $A$ is connected, $\Bar(A)$ is trivial in degree $(-1)$ and $\Cyc(A;\bar{A})$ is trivial in degree $0$. Thus there are no coboundaries in those degrees and we have compatible inclusions
    \[
    \xymatrix{
    0\ar[r] & \H^0(\Cyc(A)) \ar[r] \ar@{..>}[d]_-{\exists} & \H^0(\Bar(A)) \ar[r]^-{d} \ar@{^(->}[d] & \H^1(\Cyc(A;\bar{A})) \ar@{^(->}[d] \\
    0 \ar[r] & \Bar^0(A)^\sigma  \ar[r] & \Bar^0(A) \ar[r]^-{\iota(\sigma-1)} & \Cyc^1(A;\bar{A})
}
    \]
    where the top row is exact by Lemma \ref{lem:comparing cyclic and bar} and the bottom row is exact since $\iota\maps \Bar^0(A)\to \Cyc^1(A;\bar{A})$
    is an isomorphism.   
    An easy diagram chase shows that there is a natural map filling-in the diagram with the dotted arrow. This exhibits the claimed square, which we next show to be a pullback.
    
    Since we have a presentation
    \[
    \H^0(\Bar(A)^\sigma) = \ker(d_{\Bar}) \cap \ker(\sigma-1) \subseteq \Bar^0(A),
    \]
    this cohomology of $\Bar(A)^\sigma$ is already the pullback of the two kernel inclusions into $\Bar(A)$. So there is a natural map $\tilde{\tau}\maps \H^0(\Cyc(A))\dashrightarrow \H^0(\Bar(A)^\sigma)$. But the composition $$\Bar(A)^\sigma\xrightarrow{\tau} \Cyc(A)\to \Bar(A)$$  is the natural inclusion of $\sigma$-invariants, so $\tau\maps \H^0(\Bar(A)^\sigma) \to \H^0(\Cyc(A))$ is an inverse to $\tilde{\tau}$, and thus $\H^0(\Cyc(A))$ satisfies the universal property of the pullback as well.
\end{proof}

Lastly, we often wish to compare the cohomologies of different algebras. For this we have the following.
\begin{lemma}\label{lem:H1-injective}
    Let $f\maps A\to B$ be a homomorphism of homologically connected augmented dg algebras with no $R$-torsion. If $f$ induces an injection $\H^1(A)\to \H^1(B)$, then the maps
    \[
    \H^0(\Cyc(A))\to \H^0(\Cyc(B)) \text{ and } \H^0(\Bar(A))\to \H^0(\Bar(B))
    \]
    are injections as well.
\end{lemma}
\begin{proof}
        We show that $\H^0(\Bar(A))\to \H^0(\Bar(B))$ is injective by considering the induced map on spectral sequences associated to the filtration of $\Bar(-)$ by tensor degree, used in the proof of quasi-isomorphism invariance. On its $E_1$-page we get the map $\Bar(H(A))\to \Bar(H(B))$ by the K\"{u}nneth isomorphism, which in total degree $0$ is the map of tensor (co)algebras $T(H^1(A))\to T(H^1(B))$ and is thus injective.

    Since there are no elements in strictly negative degrees, no differentials can hit the terms in total degree $0$ and the $E_{\infty}$-pages include into the $E_1$-pages:
    \[
\xymatrix{
E_\infty^0(A) \ar@{^(->}[r] \ar[d] & E_1^0(A) \ar@{^(->}[d]^{\text{$\H^1$-injectivity}} \\
E^0_{\infty}(B) \ar@{^(->}[r] & E_1^0(B).
}
\]
showing that the induced map on $E_{\infty}$ is already injective. It follows that the map on $0$-th cohomology is similarly injective.

The same argument works to show that $\H^0(\Cyc(A))\to \H^0(\Cyc(B))$ is injective. Alternatively, one can use the commutative square
\[
\xymatrix{
\H^0(\Cyc(A)) \ar[r] \ar[d] & \H^0(\Cyc(B)) \ar[d] \\
\H^0(\Bar(A)) \ar[r] & \H^0(\Bar(B))
}
\]
in which all but the top arrow were shown to be injective. It follows that the top arrow is injective as well.
\end{proof}

The object $H^0(\Cyc(A))$ gives rise to explicit invariants of words in fundamental groups, as follows.

\section{Letter braiding invariants of loops}
Given a pointed simplicial set (or topological space) $Y$, let $C^*(Y;R)$ denote its simplicial (or singular) cochain dg algebra equipped with augmentation by restriction to the basepoint.
\begin{notation}
The Cyclic Bar complex of $Y$ will be denoted by
\[
\Cyc(Y;R) := \Cyc(C^*(Y;R)).
\]
Its cohomology $\H^*(\Cyc(Y;R))$ will be denoted by $\H^*_{\Cyc}(Y;R)$.

Similarly, the linear Bar complex of $Y$ will be denoted by
\[
\Bar(Y;R) := \Bar(C^*(Y;R)) \quad ( = \Cyc(C^*(Y;R); R ) )
\]
and its cohomology will be denoted by $\H^*_{\Bar}(Y;R)$.
\end{notation}
Since the cyclic and linear Bar constructions are invariant under quasi-isomorphisms, any homotopy equivalence of spaces induces an isomorphism on $\H^*_{\Cyc}$ and $\H^*_{\Bar}$.

In light of this, \emph{Letter braiding} is a process by which one attaches invariants to loops $S^1\to Y$, as follows.
First, Example \ref{ex:wedge of circles} shows that the circle $S^1$ has
\[
H^0_{\Cyc}(S^1;R) \cong R[t].
\]
Therefore, a loop $\ell \maps S^1\to Y$ defines a pullback
\[
\ell^*\maps H^0_{\Cyc}(Y;R) \to H^0_{\Cyc}(S^1;R) \xrightarrow{\sim} R[t],
\]
sending every element $T\in H^0_{\Cyc}(Y;R)$ to a polynomial $\ell_T(t)\in R[t]$.
\begin{definition}[Cyclic letter-braiding]\label{def:letter braiding}
    Let $\ell\maps S^1 \to Y$ be a pointed loop. For every $T\in \H^0_{\Cyc}(Y;R)$ we define the $T$-th \emph{letter braiding invariant} of $\ell$ to be the \underline{linear coefficient} of the polynomial $\ell_T(t)$ determined by $\ell^*(T)\in \H^0_{\Cyc}(S^1;R)\cong R[t]$.
\end{definition}
Quasi-isomorphism invariance implies further that letter braiding invariants only depends on the homotopy class $[\ell]\in \pi_1(Y)$.

The analogous definition with $\H^0_{\Bar}(Y;R)$ in place of $\H^0_{\Cyc}(Y;R)$ was the subject of our prequel \cite{gadish_letter-braiding_2023}.
In the case of a wedge of circles, as in Example \ref{ex:wedge of circles}, the output of this theory is invariants of words in the free group $F_n$, explicitly given in Equation \eqref{eq:letter braiding}.

\begin{prop}\label{prop:cyclic invariants through linear invariants}
    The cyclic letter braiding invariant for $T\in \H^0_{\Cyc}(Y;R)$ coincides with the analogous linear letter braiding invariants of the image of $T$ in $\H^0_{\Bar}(Y;R)$.
\end{prop}
\begin{proof}
    Let $\ell\maps S^1\to Y$ be a pointed loop. By functoriality of the Bar constructions we have a commutative diagram
    \[
    \xymatrix{
    \H^0_{\Cyc}(Y;R) \ar[r]^{\ell^*} \ar[d] & \H^0_{\Cyc}(S^1;R) \ar[r]^-{\cong} \ar[d] & R[t] \ar@{=}[d] \\
        \H^0_{\Bar}(Y;R) \ar[r]^{\ell^*} & \H^0_{\Bar}(S^1;R) \ar[r]^-{\cong} & R[t]
    }
    \]
    showing that the same polynomial results whether $T$ is sent to $\H^0_{\Cyc}(S^1;R)$ or to $\H^0_{\Bar}(S^1;R)$.
\end{proof}
It follows that cyclic letter braiding invariants are merely a subset of the linear letter braiding invariants. In particular, formulae and algorithms for their calculation are discussed in \cite{gadish_letter-braiding_2023}.

With the analysis of the previous section we can now prove Theorem \ref{thm:topological characterization}.
Recall that we let $F_S$ denote the free group generated by basis $(s\in S)$. Then homomorphisms $F_S\to R$ are in natural bijection with the set of functions $R^S = \{h\maps S\to R \}$, such as those considered in Equation \eqref{eq:letter braiding}. 

\begin{proof}[Proof of Theorem \ref{thm:topological characterization}]
    Let $Y$ be a pointed connected topological space (or, given a simplicial set take $Y$ to be its geometric realization). Construct a simplicial set $X$ with a simplicial approximation map $|X|\to Y$ such that
    \begin{itemize}
        \item the $0$-skeleton of $X$ consists only of the basepoint,
        \item the nondegenerate $1$-simplices of $X$ are in bijection with the set $S$, and for each $s\in S$ the corresponding loop in $|X|$ maps to $Y$ so as to represent the homotopy class $s\in \pi_1(Y)$, and
        \item the map $|X|\to Y$ is a weak homotopy equivalence.
    \end{itemize} 
    By the quasi-isomorphism invariance, the induced map $\H^0_{\Cyc}(Y;R) \to \H^0_{\Cyc}(X;R)$ is an isomorphism, so it suffices to prove the result for $X$.
    
    Since the $0$-skeleton of $X$ is the basepoint, its simplicial cochain algebra $C^*_{\Delta}(X;R)$ is connected. Thus Proposition \ref{prop:pullback square connected} shows that $\H^0_{\Cyc}(X;R)$ is the pullback of
    \[
    \xymatrix{
    & Bar^0(C^*(X))^\sigma \ar@{^(->}[d] \\
    \H^0_{\Bar}(X;R) \ar@{^(->}[r] & \Bar^0(C^*(X;R)).
    }
    \]
    By definition $C^1(X;R) = R^S$, so $\H^0_{\Cyc}(X;R)$ is identified with the collection of cycle-invariant tensors of functions $h\in R^S$ that also lie in $\H^0_{\Bar}(X;R)$.

    By \cite[Theorem 6.3]{gadish_letter-braiding_2023}, the restriction to the $1$-skeleton $\H^0_{\Bar}(X;R)\to \H^0_{Bar}(X^{\leq 1};R)$ has image consisting exactly of those tensors whose letter braiding invariants on the free group $F_S$ that descend to well-defined functions on $\pi_1(X)$. If the tensor happens to also be cycle-invariant, then it already lies in the image of $\H^0_{\Cyc}(X;R)$.
\end{proof}

Lastly, we relate cycle-invariance to conjugacy in the fundamental group.
\begin{lemma}\label{lem:cycle vs conjugation}
    Let $Y_S = \bigvee_{s\in S} S^1$ be a wedge of circles. Every tensor $T\in \H^0_{Bar}(Y_S;R) = T(\desus R^S)$ defines a function $I_T\maps F_S\to R$ via letter-braiding, as in Equation \eqref{eq:letter braiding}. Then the following are equivalent:
    \begin{enumerate}
        \item    $T$ is cycle-invariant.
        \item $I_T$ is a conjugation invariant function on $F_S$, namely, a class function.
    \end{enumerate}
\end{lemma}
\begin{proof}
Let $[h_1|\ldots|h_n]$ be a tensor of functions $h_i\maps S\to R$. Its letter-braiding invariant $I_{[h_1|\ldots|h_n]}\maps F_S\to R$ extends to the group ring $R[F_S]$ linearly. Then by \cite[Propsition 5.6(2)]{gadish_letter-braiding_2023}, for every $m$-tuple of generators $s_1,\ldots,s_k\in S$ we have
\[
I_{[h_1|\ldots|h_n]}((s_1-1)(s_2-1)\cdots (s_k-1)) = \begin{cases}
    h_1(s_1)h_2(s_2)\cdots h_n(s_n) & k=n \\
    0 & \text{otherwise}.
\end{cases}
\]
It follows that for the cyclic shift $\sigma[h_1|\ldots|h_n]=[h_n|h_1|\ldots|h_{n-1}]$ we have
\[
I_{[h_1|\ldots|h_n]}((s_1-1)(s_2-1)\cdots (s_k-1)) = I_{\sigma[h_1|\ldots|h_n]}((s_k-1)(s_1-1)\cdots (s_{k-1}-1)).
\]
Letter-braiding invariants are linear in $T$ in the sense that $I_{T_1+T_2} = I_{T_1}+I_{T_2}$, so the last equality immediately gives that
\[
I_{T}((s_1-1)(s_2-1)\cdots (s_k-1)) = I_{\sigma T}((s_k-1)(s_1-1)\cdots (s_{k-1}-1))
\]
for all tensors\footnote{An equivalent formulation is that letter braiding is $\sigma$-invariant, where we endow the group ring by a cycle operator $\sigma\maps R[F_S]\to R[F_S]$ acting on the basis of monomials $(s_1-1)\cdots (s_k-1)$ by cyclically permuting the terms to the right. This is distinct from a cyclic permutation of letters in individual group elements.}.

Now, if $w\in F_S$ is any word, $w-1$ can be written as a sum of monomials $(s_1-1)\cdots (s_k-1)$ as observed by Fox \cite{Fox_1953} (see also the proof of \cite[Theorem 6.1]{gadish_letter-braiding_2023}). Thus it follows that for every generator $s\in S$
\[
I_T((w-1)(s-1)) = I_{\sigma T}((s-1)(w-1)).
\]
Furthermore, there is an equality in the group ring
\[
ws = (w-1)(s-1) + (s-1) + w,
\]
showing that 
\begin{align*}
I_T(ws) =  & \; I_{T}((w-1)(s-1)) + I_T(s - 1) + I_{T}(w) \\
=  & \; I_{\sigma T}((s-1)(w-1)) + I_{\sigma T}(s - 1) + I_{\sigma T}(w) + I_{T-\sigma T}(w)\\
=& \; I_{\sigma T}((s-1)(w-1) + (s-1) + w) + I_{T-\sigma T}(w)\\
= &\; I_{\sigma T}(sw) + I_{T-\sigma T}(w).
\end{align*}
It follows that if $T=\sigma T$ then $I_T(ws)=I_T(sw)$ for all generators $s\in S$, so the function $I_T$ is indeed conjugation invariant using $sw = s(ws)s^{-1}$ and induction on word length.

Conversely, suppose $I_T$ is conjugation invariant, so $I_T(sw) = I_T(ws)$ for all $w,s\in F_S$. We apply $I_T$ to the following expansion in the group ring
$$(s_1-1)\cdots (s_k-1) = \sum_{r=0}^{k-1} (-1)^{k-1-r}\sum_{2\leq i_1< \ldots< i_r \leq k} (s_1 s_{i_1}\ldots s_{i_r})-(s_{i_1}\ldots s_{i_r})$$
obtained by expanding $(s_2-1)\cdots (s_k-1)$ into a sum of words then multiplying by $(s_1-1)$.
Focusing on the summands consisting of words that start with $s_1$, conjugation by $s_1$ moves the first letter $s_1$ to the end of the word. Thus, conjugation invariance of $I_T$ shows that $I_T((s_1-1)\cdots (s_k-1))$
\begin{eqnarray*}
&=& \sum_{r=0}^{k-1} (-1)^{k-1-r}\sum_{1< i_1< \ldots< i_r \leq k} I_T(s_1 s_{i_1}\ldots s_{i_r})-I_T(s_{i_1}\ldots s_{i_r}) \\
&=& \sum_{r=0}^{k-1} (-1)^{k-1-r}\sum_{1< i_1< \ldots< i_r \leq k} I_T(s_{i_1}\ldots s_{i_r}s_1)-I_T(s_{i_1}\ldots s_{i_r})\\
&=& I_T((s_2-1)\cdots (s_k-1)(s_1-1)).
\end{eqnarray*}
It follows that $I_T = I_{\sigma T}$ on the whole group ring $R[F_S]$, and thus $T = \sigma T$ by the uniqueness of letter braiding functions \cite[Theorem 1.1(4)]{gadish_letter-braiding_2023}. 
\end{proof}
A categorical reformulation of the last lemma is the following.
\begin{cor}\label{cor:pullback for free grps}
    For $Y_S=\bigvee_{s\in S}S^1$, letter braiding defines a pullback square
    \[
    \xymatrix{
    \H^0_{\Cyc}(Y_S;R) \ar[r] \ar[d] & \H^0_{\Bar}(Y_S;R) \ar[d] \\
    \operatorname{Func}(F_S, R)^{F_S} \ar[r] & \operatorname{Func}(F_S,R),
    }
    \]
    where $F_S$ acts on functions $f\maps F_S\to R$ by conjugation $(w.f)(x) = f(wxw^{-1})$.
\end{cor}
\begin{proof}
    Proposition \ref{prop:pullback square connected} identifies $\H^0_{\Cyc}(Y_S;R)$ with the cycle-invariant tensors in $\H^0_{\Bar}(Y_S;R) = T(\desus R^S)$. Thus the last lemma shows that the letter braiding invariants of elements in the image of $\H^0_{\Cyc}(Y_S;R)$ induce conjugation-invariant functions on $F_S$. And by Proposition \ref{prop:cyclic invariants through linear invariants}, the cyclic letter-braiding invariants of $\H^0_{\Cyc}(-;R)$ can be computed by factoring through $\H^0_{\Bar}(-;R)$. Thus the claimed square is indeed defined and commutative, and isomorphic to
    \[
    \xymatrix{
    T(\desus R^S)^\sigma \ar[r] \ar[d] & T(\desus R^S) \ar[d]\\
    \operatorname{Func}(R[F_S],R)^{F_S} \ar[r] &\operatorname{Func}(R[F_S],R).
    }
    \]
    But now Lemma \ref{lem:cycle vs conjugation} is exactly the statement that this square is a pullback.
\end{proof}

With this, we have all the pieces to prove the main theorem.

\begin{proof}[Proof of Theorem \ref{thm:main pullback}]
Let $Y$ be any pointed connected space or simplicial set. By simplicial approximation and quasi-isomorphism invariance of all complexes involved, we may assume that $Y$ is a simplicial set with $1$-skeleton $Y_S = \bigvee_{s\in S} S^1$ as in the proof of the Theorem \ref{thm:topological characterization}.

By naturality of letter braiding for the inclusion $Y_S\into Y$ we have a commutative cube
\[
    \begin{tikzcd}[row sep=1.5em, column sep = 1.5em]
    \H^0_{\Cyc}(Y;R) \arrow[rr, hookrightarrow,near start, "(1)"] \arrow[dr, hookrightarrow,near start, "(2)"] \arrow[dd] &&
    \H^0_{\Cyc}(Y_S;R) \arrow[dd] \arrow[dr,hookrightarrow,near start, "(2)"] \\
    & \H^0_{\Bar}(Y;R) \arrow[rr, hookrightarrow,near start, "(1)"] \arrow[dd, hookrightarrow, near start, "(4)"] &&
    \H^0_{\Bar}(Y_S;R) \arrow[dd, hookrightarrow, near start, "(4)"] \\
    \operatorname{Func}(\pi_1(Y),R)^{\pi_1(Y)} \arrow[rr, hookrightarrow,near start, "(3)"] \arrow[dr] && \operatorname{Func}(F_S,R) ^{F_S}\arrow[dr] \\
    & \operatorname{Func}(\pi_1(Y),R) \arrow[rr, hookrightarrow,near start, "(3)"] && \operatorname{Func}(F_S,R).
    \end{tikzcd}
\]
Indeed, Lemma \ref{lem:cycle vs conjugation} gives that cyclic letter braiding for $F_S$ lands in class functions, and thus it follows from the injectivity of the arrows labeled by (3) that $\H^0_{\Cyc}(Y;S)$ also lands in class functions of $\pi_1(Y)$. 

The statement we are proving is that the left-most wall of the cube is a pullback square.
This follows since sufficiently many of the other walls of the cube are already known to be pullbacks (the right-most wall by Corollary \ref{cor:pullback for free grps}, and the top wall by Proposition \ref{prop:pullback square connected}). In fact, all walls of the cube are pullbacks. But let us walk through the diagram chase.

Suppose the (non-cyclic) letter braiding invariant defined by $T\in \H^0_{\Bar}(Y)$ is a class function on $\pi_1(Y)$. Restricting $T$ to the $1$-skeleton of $Y$, we get a class function on $F_S$. But then, Corollary \ref{cor:pullback for free grps} shows that this restriction of $T$ admits a unique lift $\tilde{T}\in \H^0_{\Cyc}(Y_S;R)$.

So now we are dealing with the top wall of the cube, which we claim is a pullback square. Indeed, since the algebra $C^*_{\Delta}(Y;R)$ is connected and $C^1(Y;R) = C^1(Y_S;R)$, it follows that 
$$\H^0_{\Bar}(Y_S;R) = \Bar^0(Y_S;R) = T(\desus C^1(Y_S;R)) = T(\desus C^1(Y;R))=\Bar^0(Y;R) ,$$ so the restriction $\H^0_{Bar}(Y;R)\into \H^0_{\Bar}(Y_S;R)$ is isomorphic to the inclusion of $0$-cocycles $\ker(d_{\Bar})\into \Bar^0(Y;R)$. For the same reason, there is a similar equality for the cyclic complexes
$$\H^0_{\Cyc}(Y_S;R) \cong \Bar^0(Y_S;R)^\sigma = \Bar^0(Y;R)^\sigma$$ 
where isomorphism on the left is given by Proposition \ref{prop:pullback square connected} applied to $C^*_{\Delta}(Y_S;R)$. It follows that the top-back inclusion of the cube is isomorphic to the map $\H^0_{\Cyc}(Y;R) \to \Bar^0(Y;R)^\sigma$ appearing in Proposition \ref{prop:pullback square connected}. Together, these observations identify the top wall of the cube with the pullback square given in
Proposition \ref{prop:pullback square connected}, which shows that there exists a unique $\hat{T}\in \H^0_{\Cyc}(Y;R)$ lifting both $T$ and $\tilde{T}$. Injectivity of all arrows shows that $\hat{T}$ realizes the class function from which we started via cyclic letter braiding.
\end{proof}
With this, the finitely generated case of Theorem \ref{thm-main} follows quickly.
\begin{proof}[Proof of Theorem \ref{thm-main}]
    For finitely generated fundamental groups, we showed in \cite[Theorem 1.1(4)]{gadish_letter-braiding_2023} that letter braiding defines an isomorphism onto the finite type functions
    \[    
    \H^0_{\Bar}(Y;R) \xrightarrow{\sim} \varinjlim_{n} \Hom_{R}\left( R[\pi_1(Y)]/I^n , R\right) \subseteq \Hom_{R}\left( R[\pi_1(Y)] , R \right).
\]
With this, Theorem \ref{thm:main pullback} gives a pullback square
\[
\xymatrix{
\H^0_{Cyc}(Y;R) \ar@{^(->}[r] \ar@{^(->}[d] & \underset{n}{\varinjlim} \;\Hom_{R}\left( R[\pi_1(Y)]/I^n , R\right) \ar@{^(->}[d] \\
\Hom_{R}\left( R[\pi_1(Y)] , R \right)^{\pi_1(Y)} \ar@{^(->}[r] & \Hom_{R}\left( R[\pi_1(Y)] , R \right)
}
\]
showing that $H^0_{\Cyc}(Y;R)$ realizes exactly the class functions of finite type.
\end{proof}

\bibliographystyle{alpha}
\bibliography{bibliography}

\begin{thebibliography}{GJP91}

\bibitem[Fox53]{Fox_1953}
R.~H. Fox.
\newblock Free differential calculus. i: Derivation in the free group ring.
\newblock {\em Annals of Mathematics. Second Series}, 57:547–560, 1953.

\bibitem[Gad23]{gadish_letter-braiding_2023}
Nir Gadish.
\newblock Letter-braiding: a universal bridge between combinatorial group theory and topology, September 2023.
\newblock arXiv:2308.13635 [math].

\bibitem[GJP91]{getzler_differential_1991}
Ezra Getzler, John D.~S. Jones, and Scott Petrack.
\newblock Differential forms on loop spaces and the cyclic bar complex.
\newblock {\em Topology}, 30(3):339--371, November 1991.

\bibitem[Jon87]{jones_cyclic_1987}
John D.~S. Jones.
\newblock Cyclic homology and equivariant homology.
\newblock {\em Inventiones mathematicae}, 87(2):403--423, June 1987.

\bibitem[Lod11]{loday_free_2011}
Jean-Louis Loday.
\newblock Free loop space and homology, October 2011.
\newblock arXiv:1110.0405 [math].

\end{thebibliography}

\end{document}